\documentclass[10pt]{article}

\usepackage{hyperref}
\usepackage{url}
\usepackage{amsmath, amsfonts, amssymb, amsthm, amstext}
\usepackage{enumitem}
\usepackage[numbib]{tocbibind}
\usepackage[calcwidth]{titlesec}

\titleformat{\section}{\mdseries\Large}{\S\thesection}{1em}{}[\vspace{-1.6em}\rule{\titlewidth}{.25pt}]
\titleformat{\subsection}{\mdseries\scshape\large}{\S\thesubsection}{1em}{}
\titleformat{\subsubsection}{\bfseries}{\S\thesubsubsection}{1em}{}

\newtheorem{theorem}{Theorem}

\newtheorem{lemma}[theorem]{Lemma}
\newtheorem{corollary}[theorem]{Corollary}

\newcommand{\R}{\ensuremath{\mathbb R}}

  \def\beq{\begin{eqnarray}} \def\eeq{\end{eqnarray}} \def\ben{\begin{enumerate}}
\def\een{\end{enumerate}}
 \def\bit{\begin{itemize}}
\def\eit{\end{itemize}}
 \def\beqs{\begin{eqnarray*}} \def\eeqs{\end{eqnarray*}} \def\bel{\begin{lemma}} \def\eel{\end{lemma}}

\newcommand{\PP}{\mathcal P}  \newcommand{\HH}{\mathcal H}  
\newcommand{\LL}{\mathcal L}

\newcommand{\beqn}{\begin{equation}}
\newcommand{\eeqn}{\end{equation}}

\theoremstyle{remark}
\newtheorem{rem}[theorem]{Remark}

\newcommand{\diam}{\mathrm{diam}}
\newcommand{\distset}{\mathrm{dist}}
\newcommand{\Lip}{\mathrm{Lip}}
\newcommand{\TT}{\mathcal{T}}
\newcommand{\W}{\mathcal{W}}

\title{Algorithms for computing the optimal Lipschitz constant of
  interpolants with Lipschitz derivative}

\author{Matthew J. Hirn\thanks{www.math.yale.edu/$\sim$mh644} \\
Yale University \\
Department of Mathematics \\
P.O. Box 208282 \\
New Haven, Connecticut 06520-8283 \\
\texttt{matthew.hirn@yale.edu}}

\begin{document}

\maketitle

\begin{abstract}
One classical measure of the quality of an interpolating
function is its Lipschitz constant. In this paper we consider
interpolants with additional smoothness requirements, in particular
that their derivatives be Lipschitz. We show that such a measure of
quality can be easily computed, giving two algorithms, one optimal in
the dimension of the data, the other optimal in the number of points
to be interpolated. 
\end{abstract}

\section{Introduction}

For an arbitrary function $g: \R^d \rightarrow \R^n$, recall that the
\emph{Lipschitz constant} of $g$ is
defined as:
\begin{equation*}
\Lip(g) \triangleq \sup_{\substack{x,y \in \R^d \\ x \neq y}}
\frac{|g(x) - g(y)|}{|x-y|},
\end{equation*}
where $|\cdot|$ is taken to be the standard Euclidean
norm. Additionally, set $\nabla g: \R^d \rightarrow \R^d$ to be the
{\it gradient} of $g$, where $\nabla g \triangleq (\frac{\partial
  g}{\partial x_1}, \ldots, \frac{\partial g}{\partial x_d})$.

Given a finite set $E \subset \R^d$ with $\#(E) = N$ and a function
  $f: E \rightarrow \R$, it is will known that the function $f$ can be
  extended to a function $F: \R^d \rightarrow \R$ such that $\Lip(F) =
  \Lip(f)$ (see the work of Whitney
  \cite{whitney:analyticExtensions1934} and McShane
  \cite{mcshane:extensionRangeFcns1934} for the original result). Such
  a function $F$ is a {\it minimal Lipschitz extension}
  of $f$, since the Lipschitz constant of $F$ cannot be lowered while
  still interpolating the function $f$. Thus to compute $\Lip(F)$, we
  must compute $\Lip(f)$. This can clearly be accomplished in $O(N^2)$
  operations. However, using the well separated pairs decomposition
  \cite{callahan:wspd1995}, one can compute a near approximation of
  $\Lip(f)$ in only $O(N\log N)$ operations. 

In this paper, we address a related problem. We assume that along with
the function values, we are also given information about the
derivatives at each point in $E$. We wish to efficiently compute the
minimal value of $\Lip(\nabla F)$, where $F: \R^d \rightarrow \R$ is a
differentiable function whose derivative is Lipschitz that
additionally interpolates the given functional and derivative information.

Let $C^{1,1}(\R^d)$ denote the space of functions mapping $\R^d$ to $\R$
whose derivatives are Lipschitz:
\begin{equation*}
C^{1,1}(\R^d) \triangleq \{ g: \R^d \rightarrow \R : \Lip(\nabla g) <
\infty \}.
\end{equation*}
Let $\PP$ denote the set of first order polynomials (i.e., affine
functions) mapping $\R^d$ to $\R$. For $F \in C^{1,1}(\R^d)$, let
$J_xF \in \PP$ denote the first order \emph{jet} of $F$ centered at
$x$, i.e., $J_xF(z) \triangleq F(x) + \nabla F(x) \cdot (z-x)$. A
\emph{Whitney 1-field} $\PP_E \triangleq \{P_x \in \PP: x \in E\}$ is a set
polynomials in $\PP$ indexed by the set $E \subset \R^d$.

In this paper we address some of the computational aspects of the following
problem:

\underline{\textsc{Jet Interpolation Problem}}: Suppose we are given a
finite set $E \subset \R^d$ and a 1-field $\PP_E = \{P_x \in \PP : x
\in E\}$. Compute a function $F \in C^{1,1}(\R^d)$ such that
\begin{enumerate}[topsep=0pt]
\item
$J_xF = P_x$ for all $x \in E$.
\item
$\Lip(\nabla F)$ is minimal.
\end{enumerate}

There are two theoretical problems tied into the \textsc{Jet
  Interpolation Problem}. The first of these involves determining the
optimal value of the semi-norm $\Lip(\nabla F)$. It is, by definition, given by:
\begin{equation*}
\LL(\PP_E) \triangleq \inf \{ \Lip(\nabla F) : F \in
C^{1,1}(\R^d) \text{ \& } J_xF = P_x \enspace \forall \, x \in E \}.
\end{equation*}
The second problem is to construct a function $F \in
C^{1,1}(\R^d)$ that interpolates the 1-field $\PP_E$ such that
$\Lip(\nabla F) = \LL(\PP_E)$.

Remarkably, there are solutions to both of these problems. In
\cite{legruyer:minLipschitzExt2009}, Le Gruyer gives a closed formula
for $\LL(\PP_E)$, while in \cite{wells:diffFuncLipDeriv1973} Wells
gives a construction for the interpolant $F$.

The two theoretical problems lead to two corresponding computational
problems: (1) efficiently computing $\LL(\PP_E)$ and (2) efficiently
computing the interpolant $F$. The theoretical results of Le Gruyer
and Wells give a roadmap by which to accomplish these tasks.

The main result of this paper is to give an
algorithm that efficiently computes a number $M$ with the same order of
magnitude of $\LL(\PP_E)$. In a follow up paper, we shall address the
problem of efficiently computing an interpolant $F \in C^{1,1}(\R^d)$
for $\PP_E$ such that $\Lip(\nabla F) = M$.

Two numbers $X, Y$ that are dependent upon $E, \PP_E,$ and $d$ are said
to have the same \emph{order of magnitude} if there exist universal
constants $c$ and $C$ such that $cY \leq X \leq CY$.

By compute we mean develop an algorithm that can run on an idealized
computer with standard von Neumann architecture, able to work with
exact real numbers. We ignore roundoff, overflow, and underflow
errors, and suppose that an exact real number can be stored at each
memory address. Additionally, we suppose that it takes one machine
operation to add, subtract, multiply, or divide two real numbers $x$
and $y$, or to compare them (i.e., decide whether $x < y$, $x > y$, or
$x = y$).

The \emph{work} of an algorithm is the number of machine operations
needed to carry it out, and the \emph{storage} of an algorithm is the
number of random access memory addresses required.

Throughout, we shall set $\#(E) = N$ to be the number of points in
$E$.

Some related work on the computation of interpolants in $C^m(\R^d)$ is given in
\cite{fefferman:fittingDataI, fefferman:fittingDataII,
  fefferman:interLinProg2011, fefferman:nearOptC2R2I}. In
particular, this work is most closely related to
\cite{fefferman:fittingDataI, fefferman:fittingDataII}, but by
working in $C^{1,1}(\R^d)$, and using the semi-norm $\Lip(\nabla F)$
as opposed to some $C^m$ norm, we are able to achieve order of
magnitude constants that do not depend on the dimension.

\section{Computing $\LL(\PP_E)$}

In this section we present two algorithms for computing
$\LL(\PP_E)$. One is an exact computation that is simply a corollary
of the results found in \cite{legruyer:minLipschitzExt2009}; it runs
in $O(d N^2)$ time and requires $O(dN)$ storage. The second, which
requires more effort to develop, computes the order of magnitude of
$\LL(\PP_E)$ in $O(d^{d/2}N\log N)$ time and requires
$O(d^{d/2}N)$ storage.

\subsection{Closed formula for $\LL(\PP_E)$ and an efficient algorithm
in the dimension $d$}

In \cite{legruyer:minLipschitzExt2009}, Le Gruyer gives a closed
formula for $\LL(\PP_E)$, which is immensely useful for its
computation. We summarize the results in this section.

For the 1-field $\PP_E = \{P_x \in \PP : x \in E\}$, define two
functionals $A: E \times E \rightarrow [0,\infty]$ and $B: E \times E
\rightarrow [0,\infty]$,
\begin{equation*}
A(x,y) \triangleq \frac{|P_x(x) - P_y(x) + P_x(y) -
  P_y(y)|}{|x-y|^2}, \qquad B(x,y) \triangleq \frac{|\nabla P_x - \nabla P_y|}{|x-y|}.
\end{equation*}
Note that $A$ was originally formulated differently in
\cite{legruyer:minLipschitzExt2009}, we have simply rewritten it in a
form more useful for our purposes. Additionally, recall that $\PP$ is
the set of first order polynomials, so for any $P \in \PP$, $\nabla P$
is a constant vector in $\R^d$.

Using $A$ and $B$, define $\Gamma$ as:
\begin{equation*}
\Gamma(\PP_E) \triangleq \max_{\substack{x,y \in E \\ x \neq y}}
\sqrt{A(x,y)^2 + B(x,y)^2} + A(x,y).
\end{equation*}
We then have the following theorem:
\begin{theorem}[Le Gruyer,
  \cite{legruyer:minLipschitzExt2009}] \label{thm: le gruyer}
For any finite $E \subset \R^d$ and any 1-field $\PP_E$,
\begin{equation*}
\LL(\PP_E) = \Gamma(\PP_E).
\end{equation*}
\end{theorem}
Thus the functional $\Gamma(\PP_E)$ is the closed form of $\LL(\PP_E)$. If
the number of data points $N$ is reasonable, then it yields an obvious
algorithm for computing $\LL(\PP_E)$ by simply evaluating $A(x,y)$ and
$B(x,y)$ for all unique pairs $x,y \in E$ and computing $\Gamma(\PP_E)$. We
state this as a corollary.
\begin{corollary} \label{cor: le gruyer alg}
There is an algorithm, whose inputs are the set $E$ and the 1-field
$\PP_E$, that computes $\LL(\PP_E)$ exactly. It requires $O(dN^2)$
work and $O(dN)$ storage.
\end{corollary}
The obvious benefit of this algorithm is that it computes $\LL(\PP_E)$
exactly. Additionally, the storage is asymptotically optimal both in
$d$ and in $N$, and the work is asymptotically optimal in $d$. On the
other hand, if the number of points $N$ is large, then the $O(dN^2)$
work is at best impractical, and at worst impossible. In order to
handle this situation, we turn to the well separated pairs
decomposition.
\begin{rem}
When we say that we input $\PP_E$ into the computer, what we 
mean is that we input $P_x(x) \in \R$ and $\nabla P_x \in \R^d$ for
each $x \in E$.
\end{rem}
\begin{rem}
In fact Theorem \ref{thm: le gruyer} holds not only for $\R^d$, but
for any Hilbert space with real valued inner product. Consequently,
Corollary \ref{cor: le gruyer alg} can be applied to work in any
Hilbert space (replacing the Euclidean norm with the Hilbert space
norm), including infinite dimensional Hilbert spaces, so long
as one has a method (or ``black box'') by which to compute inner
products. This is often the case when the set $E \subset \R^d$ but 
one utilizes a kernel function $k: E \times E \rightarrow \R$ such
that $k(x,y)$ is the inner product in a Hilbert space $\HH$ after some
implicit mapping $\varphi: E \rightarrow \HH$. 
\end{rem}

\subsection{Well separated pairs decomposition} \label{sec: wspd}

The well separated pairs decomposition was first introduced by
Callahan and Kosaraju in \cite{callahan:wspd1995}; we shall make use
of a modified version that was described in detail in
\cite{fefferman:fittingDataI}.

First, recall the standard definitions of the \emph{diameter} of a set
and the \emph{distance} between two sets. Let $S, T \subset \R^d$,
\begin{equation*}
\diam(S) \triangleq \sup_{\substack{x,y \in S \\ x \neq y}} |x-y|,
\qquad \distset(S,T) \triangleq \inf_{\substack{x \in S \\ y \in T}} |x-y|.
\end{equation*}
Let $\varepsilon > 0$; two sets $S, T \subset \R^d$ are
$\varepsilon$\emph{-separated} if
\begin{equation*}
\max \{\diam(S), \diam(T)\} < \varepsilon \distset(S,T).
\end{equation*}

We follow the construction detailed by Fefferman and Klartag in
\cite{fefferman:fittingDataI}. Let $\TT$ be a collection of subsets of
$E$. For any $\Lambda \subset
\TT$, set
\begin{equation*}
\cup \Lambda \triangleq \bigcup_{S \in \Lambda} S = \{x : x \in S
\text{ for some } S \in \Lambda \}.
\end{equation*}
Let $\W$ be a set of pairs $(\Lambda_1, \Lambda_2)$ where $\Lambda_1,
\Lambda_2 \subset \TT$. For any $\varepsilon > 0$, the pair $(\TT, \W)$
is an \emph{$\varepsilon$-well separated pairs decomposition} or
\emph{$\varepsilon$-WSPD} for short if the following properties hold:
\begin{enumerate}[topsep=0pt]
\item \label{item: F-K 1}
$\bigcup_{(\Lambda_1,\Lambda_2) \in \W} \cup \Lambda_1 \times \cup
\Lambda_2 = \{ (x,y) \in E \times E : x \neq y \}$.
\item \label{item: F-K 2}
If $(\Lambda_1, \Lambda_2), (\Lambda_1', \Lambda_2') \in \W$ are
distinct pairs, then $(\cup \Lambda_1 \times \cup \Lambda_2) \cap
(\cup \Lambda_1' \times \cup \Lambda_2') = \emptyset$.
\item
$\cup \Lambda_1$ and $\cup \Lambda_2$ are $\varepsilon$-separated for
any $(\Lambda_1, \Lambda_2) \in \W$.
\item
$\#(\TT) < C(\varepsilon, d) N$ and $\#(\W) < C(\varepsilon, d) N$.
\end{enumerate}

As shown in \cite{fefferman:fittingDataI}, there is a data structure
representing $(\TT,\W)$ that satisfies the following additional
properties as well:
\begin{enumerate}[topsep=0pt,resume]
\item
The amount of storage to hold the data structure is
$O((\sqrt{d}/\varepsilon)^dN)$.
\item \label{item: F-K 6}
The following tasks require at most $O((\sqrt{d}/\varepsilon)^dN\log N)$ work and
$O((\sqrt{d}/\varepsilon)^dN)$ storage:
\begin{enumerate}[topsep=0pt]
\item
Go over all $S \in \TT$, and for each $S$ produce a list of elements in
$S$.
\item
Go over all $(\Lambda_1,\Lambda_2) \in W$, and for each $(\Lambda_1,
\Lambda_2)$ produce the elements (in $\TT$) of $\Lambda_1$ and
$\Lambda_2$.
\item
Go over all $S \in \TT$, and for each $S$ produce the list of all
$(\Lambda_1, \Lambda_2) \in W$ such that $S \in \Lambda_1$.
\item
Go over all $x \in E$, and for each $x \in E$ produce a list of $S \in
\TT$ such that $x \in S$.
\end{enumerate}
\end{enumerate}
As a result of property \ref{item: F-K 6}, it follows that the
following properties also hold:
\begin{enumerate}[topsep=0pt,resume]
\item \label{item: F-K 7}
For $C(\varepsilon,d) = O((\sqrt{d}/\varepsilon)^d)$,
\begin{enumerate}[topsep=0pt]
\item
$\sum_{(\Lambda_1, \Lambda_2) \in \W} (\#(\Lambda_1) + \#(\Lambda_2)) <
C(\varepsilon,d) N\log N$.
\item
$\sum_{S \in \TT} \#(S) < C(\varepsilon,d) N\log N$.
\end{enumerate}
\end{enumerate}

\begin{theorem}[Fefferman and Klartag,
  \cite{fefferman:fittingDataI}] \label{thm: F-K WSPD}
There is an algorithm, whose inputs are the parameter $\varepsilon >
0$ and a subset $E \subset \R^d$ with $\#(E) = N$, that outputs a
$\varepsilon$-WSPD $(\TT,\W)$ of $E$ such that properties \ref{item:
  F-K 1},$\ldots$,\ref{item: F-K 7} hold. The algorithm requires
$O((\sqrt{d}/\varepsilon)^dN\log N)$ work and
$O((\sqrt{d}/\varepsilon)^d) N)$ storage.
\end{theorem}

\begin{rem}
The algorithm presented in \cite{fefferman:fittingDataI} is built upon
the well separated pairs decomposition algorithm developed by Callahan
and Kosaraju in \cite{callahan:wspd1995}. In fact, $\TT$ is a
completely balanced binary tree based off the inorder relation derived
from the fair split tree presented in \cite{callahan:wspd1995}. In
particular, $\#(\TT) < 2N$ and the height of the tree is bounded by
$\lceil \log_2 N \rceil + 1$. The list $\W$ has a one-to-one
correspondence with the well separated pair list presented in
\cite{callahan:wspd1995}, hence $\#(\W) = O((\sqrt{d}/\varepsilon)^dN)$.
\end{rem}

\subsection{Efficient computation of $\LL(\PP_E)$ in the number of
  points $N$}

In this section we prove the following theorem:
\begin{theorem}\label{thm: efficient L(PE)}
There is an algorithm, whose inputs are the set $E$ and the 1-field
$\PP_E$, that computes the order of magnitude of $\LL(\PP_E)$. It
requires $O(d^{d/2}N\log N)$ work and $O(d^{d/2}N)$ storage.
\end{theorem}

The plan for proving Theorem \ref{thm: efficient L(PE)} is the
following. First we view Le Gruyer's $\Gamma$ functional from the
perspective of the classical Whitney conditions. Once we formalize
this concept, we can use the $\varepsilon$-WSPD of Fefferman and
Klartag, since they built it to handle interpolants in $C^m(\R^n)$
satisfying Whitney conditions.

Concerning the first part, consider the original Whitney conditions
for $C^{1,1}(\R^n)$:
\begin{enumerate}[topsep=0pt]
\setcounter{enumi}{-1}
\renewcommand{\theenumi}{$(W_{\arabic{enumi}})$}
\renewcommand{\labelenumi}{\theenumi}
\item\label{W0}
$|(P_x-P_y)(x)| \leq M|x-y|^2$ for all $x,y \in E$.
\item\label{W1}
$|\frac{\partial}{\partial x_i}(P_x - P_y)(x)| \leq M|x-y|$ for all $x,y \in E$,
$i = 1, \ldots, d$.
\end{enumerate}
Whitney's extension theorem states that if \ref{W0} and \ref{W1} hold,
then there exists an $F \in C^{1,1}(\R^d)$ that interpolates $\PP_E$ such
that $\Lip(\nabla F) \leq C(d)M$. 

The main contribution of \cite{legruyer:minLipschitzExt2009} is to refine \ref{W0} and \ref{W1}
such that $C(d) = 1$; this is $\Gamma$. Indeed, the functional $A$
corresponds to \ref{W0}, the functional $B$ corresponds to
\ref{W1}, and $\Gamma$ pieces them together. Note there are some
small, but significant differences. In particular, the functional $A$
is essentially a symmetric version of \ref{W0}; using one is
equivalent to using the other, up to a factor of two. The functional
$B$ though, merges all of the partial derivative information into one
condition, unlike \ref{W1}. Thus they are equivalent only up to a
factor of $d$, the dimension of the Euclidean space we are working
in. For the algorithm in this section, we will use the functional $B$
since it is both simpler and more useful than \ref{W1}, but use \ref{W0}
instead of $A$. Additionally, we will treat them separately instead of together like in
$\Gamma$; Lemma \ref{lem: first gamma estimate} contains the details.

For the 1-field $\PP_E$,
define the functional $\widetilde{A}: E \times E \rightarrow
[0,\infty]$ (which is essentially the same as \ref{W0}),
\begin{equation*}
\widetilde{A}(x,y) \triangleq \frac{|P_x(x) - P_y(x)|}{|x-y|^2}.
\end{equation*}
Additionally, set
\begin{equation*}
\widetilde{\Gamma}(\PP_E) \triangleq \max_{\substack{x,y \in E \\ x
    \neq y}} \Big\{ \max\{\widetilde{A}(x,y), B(x,y)\} \Big\}.
\end{equation*}
The functional $\widetilde{\Gamma}(\PP_E)$ is more easily approximated
via the $\varepsilon$-WSPD than $\Gamma(\PP_E)$. Furthermore, as the
following Lemma shows, they have the same order of magnitude.
\begin{lemma} \label{lem: first gamma estimate}
For any finite $E \subset \R^d$ and any 1-field $\PP_E$,
\begin{equation*}
\widetilde{\Gamma}(\PP_E) \leq \Gamma(\PP_E) \leq 2(1+ \sqrt{2})
\widetilde{\Gamma}(\PP_E).
\end{equation*}
\end{lemma}

\begin{proof}
To bridge the gap between $\Gamma(\PP_E)$ and
$\widetilde{\Gamma}(\PP_E)$, we first consider
\begin{equation*}
\Gamma'(\PP_E) \triangleq \max_{\substack{x,y \in E \\ x \neq y}}
\Big\{ \max \{A(x,y), B(x,y)\} \Big\}.
\end{equation*}
Clearly $\Gamma'(\PP_E) \leq \Gamma(\PP_E)$. Furthermore,
\begin{align*}
\Gamma(\PP_E) &= \max_{\substack{x,y \in E \\ x \neq y}} \sqrt{A(x,y)^2
  + B(x,y)^2} + A(x,y) \\
&\leq \sqrt{\Gamma'(\PP_E)^2 + \Gamma'(\PP_E)^2} + \Gamma'(\PP_E) \\
&\leq (1 + \sqrt{2})\Gamma'(\PP_E).
\end{align*}
Thus $\Gamma(\PP_E)$ and $\Gamma'(\PP_E)$ have the same order of
magnitude, and in particular,
\begin{equation}\label{eqn: Gamma and Gamma'}
\Gamma'(\PP_E) \leq \Gamma(\PP_E) \leq (1+\sqrt{2})\Gamma'(\PP_E).
\end{equation}

Now let us consider $\Gamma'(\PP_E)$ and
$\widetilde{\Gamma}(\PP_E)$ (which means considering
$A(x,y)$ and $\widetilde{A}(x,y)$). First,
\begin{align*}
|P_x(x) - P_y(x) + P_x(y) - P_y(y)| &\leq |P_x(x) - P_y(x)| + |P_x(y)
- P_y(y)| \\
&\leq 2\widetilde{\Gamma}(\PP_E) |x-y|^2,
\end{align*}
and so, $\Gamma'(\PP_E) \leq 2\widetilde{\Gamma}(\PP_E)$. For a
reverse inequality, we note,
\begin{equation*}
|P_x(x) - P_y(x) + P_x(y) - P_y(y)| = |2(P_x(x) - P_y(x)) + (\nabla
P_y - \nabla P_x) \cdot (x-y)|.
\end{equation*}
Thus,
\begin{align*}
2|P_x(x) - P_y(x)| &\leq \Gamma'(\PP_E)|x-y|^2 + |(\nabla P_y - \nabla
P_x) \cdot (x-y)| \\
&\leq 2\Gamma'(\PP_E) |x-y|^2,
\end{align*}
which yields $\widetilde{\Gamma}(\PP_E) \leq
\Gamma'(\PP_E)$. Combining the two inequalities,
\begin{equation}\label{eqn: Gamma' and Gammatilde}
\widetilde{\Gamma}(\PP_E) \leq \Gamma'(\PP_E) \leq 2\widetilde{\Gamma}(\PP_E).
\end{equation}
Putting \eqref{eqn: Gamma and Gamma'} and \eqref{eqn: Gamma' and
  Gammatilde} together completes the proof.
\end{proof}

We will also need the following simple Lemmas.
\begin{lemma}\label{lem: wspd estimates}
Let $(\TT,\W)$ be a $\varepsilon$-WSPD, $(\Lambda_1, \Lambda_2) \in
\W$, $x,x',x'' \in \cup \Lambda_1$, and $y,y' \in \cup
\Lambda_2$. Then,
\begin{align*}
&|x'-x''| \leq \varepsilon |x-y| \\
&|x'-y'| \leq (1+2\varepsilon) |x-y|.
\end{align*}
\end{lemma}
\begin{proof}
Use the definition of $\varepsilon$-separated.
\end{proof}

\begin{lemma}\label{lem: polynomial shift}
Suppose that $P \in \PP$, $x \in \R^d$, $\delta > 0$, and $M > 0$
satisfy
\begin{align*}
&|P(x)| \leq M \delta^2 \\
&|\nabla P| \leq M \delta.
\end{align*}
Then, for any $y \in \R^d$,
\begin{equation*}
|P(y)| \leq M(\delta + |x-y|)^2.
\end{equation*}
\end{lemma}
\begin{proof}
Using Taylor's Theorem,
\begin{align*}
|P(y)| &= |P(x) + \nabla P(x) \cdot (y-x)| \\
&\leq |P(x)| + |\nabla P| |x-y| \\
&\leq M\delta^2 + M\delta |x-y| \\
&\leq M(\delta + |x-y|)^2.
\end{align*}
\end{proof}

\begin{proof}[Proof of Theorem \ref{thm: efficient L(PE)}]
In order to simplify notation, let $\widetilde{\Gamma}(x,y)$ denote
the quantity maximized in the definition of
$\widetilde{\Gamma}(\PP_E)$, i.e.,
\begin{equation*}
\widetilde{\Gamma}(x,y) = \max\{\widetilde{A}(x,y), B(x,y)\}.
\end{equation*}
Additionally, set
\begin{equation*}
\widetilde{A}(\PP_E) \triangleq \max_{\substack{x,y \in E \\ x \neq
    y}} \widetilde{A}(x,y), \qquad B(\PP_E) \triangleq \max_{\substack{x,y \in E \\ x \neq y}} B(x,y).
\end{equation*}

Our algorithm works as follows. For now, let $\varepsilon > 0$ be
arbitrary and invoke the algorithm from Theorem \ref{thm: F-K
  WSPD}. This gives us an $\varepsilon$-WSPD $(\TT,\W)$ in
$O((\sqrt{d}/\varepsilon)^dN\log N)$ work and using
$O(\sqrt{d}/\varepsilon)^d N)$ storage. For each $(\Lambda_1,
\Lambda_2) \in \W$, pick at random a representative $(x_{\Lambda_1},
x_{\Lambda_2}) \in \cup \Lambda_1 \times \cup
\Lambda_2$. Additionally, for each $S \in \TT$, pick at random a
representative $x_S \in S$.

Now compute the following:
\begin{align*}
\widetilde{\Gamma}_1 &\triangleq \max_{(\Lambda_1, \Lambda_2) \in \W}
\widetilde{\Gamma}(x_{\Lambda_1}, x_{\Lambda_2}) \\
\widetilde{\Gamma}_2 &\triangleq \max_{(\Lambda_1, \Lambda_2) \in \W}
\, \max_{i=1,2} \, \max_{S \in \Lambda_i} \widetilde{\Gamma}(x_{\Lambda_i},
x_S) \\
\widetilde{\Gamma}_3 &\triangleq \max_{S \in \TT} \max_{x \in S}
\widetilde{\Gamma}(x,x_S) \\
\widetilde{\Gamma}(\PP_E,\TT,\W) &\triangleq \max\{\widetilde{\Gamma}_1,
\widetilde{\Gamma}_2, \widetilde{\Gamma}_3\}.
\end{align*}
Define $\widetilde{A}(\PP_E, \TT, \W)$ and $B(\PP_E, \TT, \W)$
analogously. Using properties \ref{item: F-K 6} and \ref{item: F-K 7} from Section
\ref{sec: wspd}, we see that computing
$\widetilde{\Gamma}(\PP_E,\TT,\W)$ requires
$O((\sqrt{d}/\varepsilon)^dN\log N)$ work and
$O((\sqrt{d}/\varepsilon)^dN)$ storage.

Now we show that $\widetilde{\Gamma}(\PP_E,\TT,\W)$ has the same order
of magnitude as $\widetilde{\Gamma}(\PP_E)$. Clearly,
$\widetilde{\Gamma}(\PP_E, \TT, \W) \leq
\widetilde{\Gamma}(\PP_E)$. For the other inequality, we break
$\widetilde{\Gamma}$ into its two parts, noting that
$\widetilde{\Gamma}(\PP_E) = \max\{\widetilde{A}(\PP_E), B(\PP_E)\}$ and
  $\widetilde{\Gamma}(\PP_E, \TT, \W) = \max\{\widetilde{A}(\PP_E, \TT,
  \W), B(\PP_E, \TT, \W)\}$. Thus we can work with $\widetilde{A}$ and
  $B$ separately.

The functional $B$ is simply the Lipschitz constant of the mapping $x
\mapsto \nabla P_x$. It is known that 
\begin{equation} \label{eqn: B approx}
B(\PP_E) \leq (1+C\varepsilon) B(\PP_E, \TT, \W).
\end{equation}
See for example Proposition 2 of
\cite{fefferman:smthIntEffAlg2013}. Using the particular
construction in this proof, we can take $C = 6$.

We now turn to $\widetilde{A}$. Let $x,y \in E$, $x \neq y$. By
properties \ref{item: F-K 1} and \ref{item: F-K 2} of Section
\ref{sec: wspd}, there is a unique pair $(\Lambda_1, \Lambda_2) \in
\W$ such that $(x,y) \in \cup \Lambda_1 \times \cup
\Lambda_2$. Additionally, by the definition of $(\TT,\W)$, there exists
a set $S \in \Lambda_1$ such that $x \in S$ and a set $T \in \Lambda_2$ such
that $y \in T$.

Let $M = \widetilde{\Gamma}(\PP_E,\TT,\W)$. We then have, using the
triangle inequality, the definition of $\widetilde{\Gamma}_3$, and
Lemma \ref{lem: wspd estimates},
\begin{align}
|P_x(x) - P_y(x)| &\leq |P_x(x) - P_{x_S}(x)| + |P_{x_S}(x) - P_y(x)|
\nonumber \\
&\leq \widetilde{\Gamma}_1|x-x_S|^2 + |P_{x_S}(x) - P_y(x)| \nonumber \\
&\leq \varepsilon M |x-y|^2 + |P_{x_S}(x) -
P_y(x)|. \label{eqn: A estimate 1}
\end{align}
Continuing with the second term of the right hand side of \eqref{eqn:
  A estimate 1}, we use the triangle inequality, Lemma \ref{lem:
  polynomial shift}, the definition of $\widetilde{\Gamma}_2$, and
Lemma \ref{lem: wspd estimates},
\begin{align}
|P_{x_S}(x) - P_y(x)| &\leq |P_{x_S}(x) - P_{x_{\Lambda_1}}(x)| +
|P_{x_{\Lambda_1}}(x) - P_y(x)| \nonumber \\
&\leq \widetilde{\Gamma}_2 (|x_S - x_{\Lambda_1}| + |x -
x_{\Lambda_1}|)^2 + |P_{x_{\Lambda_1}}(x) - P_y(x)| \nonumber \\
&\leq 4\varepsilon^2 M |x-y|^2 + |P_{x_{\Lambda_1}}(x) -
P_y(x)|. \label{eqn: A estimate 2}
\end{align}
Continuing with the second term of the right hand side of \eqref{eqn:
  A estimate 2}, we use the triangle inequality, Lemma \ref{lem:
  polynomial shift}, the definition of $\widetilde{\Gamma}_3$, and
Lemma \ref{lem: wspd estimates},
\begin{align}
|P_{x_{\Lambda_1}}(x) - P_y(x)| &\leq |P_y(x) - P_{x_T}(x)| +
|P_{x_T}(x) - P_{x_{\Lambda_1}}(x)| \nonumber \\
&\leq \widetilde{\Gamma}_3 (|y-x_T| + |x-y|)^2 + |P_{x_T}(x) -
P_{x_{\Lambda_1}}(x)| \nonumber \\
&\leq (1+\varepsilon)^2 M |x-y|^2 + |P_{x_T}(x) -
P_{x_{\Lambda_1}}(x)|. \label{eqn: A estimate 3}
\end{align}
Continuing with the second term of the right hand side of \eqref{eqn:
  A estimate 3}, we use the triangle inequality, Lemma \ref{lem:
  polynomial shift}, the definitions of $\widetilde{\Gamma}_1$ and
$\widetilde{\Gamma}_2$, as well as Lemma \ref{lem: wspd estimates},
\begin{align}
|P_{x_T}(x) - P_{x_{\Lambda_1}(x)}(x)| &\leq |P_{x_T}(x) -
P_{x_{\Lambda_2}}(x)| + |P_{x_{\Lambda_2}}(x) - P_{x_{\Lambda_1}}(x)|
\nonumber \\
&\leq \widetilde{\Gamma}_2(|x_T-x_{\Lambda_2}| + |x-x_{\Lambda_2}|)^2
+ \widetilde{\Gamma}_1(|x_{\Lambda_2}-x_{\Lambda_1}| +
|x-x_{\Lambda_1}|)^2 \nonumber \\
&\leq 2(1+3\varepsilon)^2 M |x-y|^2. \label{eqn: A estimate 4}
\end{align}
Putting \eqref{eqn: A estimate 1}, \eqref{eqn: A estimate 2},
\eqref{eqn: A estimate 3}, \eqref{eqn: A estimate 4} together, we get:
\begin{equation}\label{eqn: A estimate 5}
|P_x(x) - P_y(x)| \leq 3M |x-y|^2 + 23\varepsilon M |x-y|^2.
\end{equation}
Taking $\varepsilon = 1/2$ gives the desired bounds on the work and
storage, and in addition yields
\begin{equation*}
\widetilde{\Gamma}(\PP_E) \leq C \widetilde{\Gamma}(\PP_E, \TT, \W).
\end{equation*}
The proof is completed by applying Lemma \ref{lem: first gamma estimate}.
\end{proof}

\begin{rem}
Examining \eqref{eqn: B approx} and \eqref{eqn: A estimate 5}, we see that
$\widetilde{\Gamma}(\PP_E)$ and $\widetilde{\Gamma}(\PP_E,\TT,\W)$ have
the same order of magnitude with constants $c=1$ and $C =
C(\varepsilon) = 3 + 23\varepsilon$. Thus,
\begin{equation*}
\widetilde{\Gamma}(\PP_E,\TT,\W) \leq \widetilde{\Gamma}(\PP_E) \leq C(\varepsilon)
\widetilde{\Gamma}(\PP_E,\TT,\W),
\end{equation*}
Recalling Lemma \ref{lem: first gamma estimate}, we then have
\begin{equation*}
\widetilde{\Gamma}(\PP_E,\TT,\W) \leq \Gamma(\PP_E) \leq C'(\varepsilon)
\widetilde{\Gamma}(\PP_E,\TT,\W),
\end{equation*}
where $C'(\varepsilon) = 2(1+\sqrt{2})C(\varepsilon) =
2(1+\sqrt{2})(3+23\varepsilon)$. Therefore, as $\varepsilon
\rightarrow 0$, $C'(\varepsilon) \rightarrow 6(1+\sqrt{2})$.
\end{rem}

\section{Acknowledgements}

The author would like to thank Charles Fefferman for introducing him
to the problem and Hariharan Narayanan for numerous insightful
conversations.

\bibliography{/Users/matthewhirn/Dropbox/Mathematics/Bibliography/MainBib}

\end{document}